\documentclass[12pt]{article} 

\usepackage{amsmath}
\usepackage{amsmath,amscd}
\usepackage{graphicx}
\usepackage{amsthm}
\usepackage{hyperref}

\newtheorem{theorem}{Theorem}[section]
\newtheorem{lemma}[theorem]{Lemma}
\newtheorem{prop}[theorem]{Proposition}
\newtheorem{corollary}[theorem]{Corollary}
\theoremstyle{definition}
\newtheorem{definition}[theorem]{Definition}
\newtheorem{example}[theorem]{Example}

\theoremstyle{remark}
\newtheorem{remark}[theorem]{Remark}

\usepackage{geometry} 
\usepackage{amssymb}
\usepackage{graphicx} 

\usepackage{float} 
\usepackage{wrapfig} 

\usepackage{lipsum} 

\newcommand{\Aut}{\mathrm{Aut}}
\newcommand{\Hom}{\mathrm{Hom}}

\graphicspath{{./Pictures/}} 

\long\def\MSC#1\EndMSC{\def\arg{#1}\ifx\arg\empty\relax\else
	{\par\narrower\noindent%
		2000 Mathematics Subject Classification: #1\par}\fi}

\begin{document}

\title{$G$-marked moduli spaces }
\date{}
\author{Binru Li}
\maketitle

\begin{abstract}
\noindent
The aim of this paper is to investigate the closed subschemes of moduli spaces corresponding to projective varieties which admit an effective action by a given finite group $G$. To achieve this, we introduce the moduli functor $\mathsf{M}^G_h$ of $G$-marked Gorenstein canonical models with Hilbert polynomial $h$, and prove the existence of $\mathfrak{M}_h[G]$, the coarse moduli scheme for $\mathsf{M}^G_h$. Then we show that $\mathfrak{M}_h[G]$ has a proper and finite morphism onto $\mathfrak{M}_h$ so that its image $\mathfrak{M}_h(G)$ is a closed subscheme.\\
In the end we obtain the canonical representation type decomposition $\mathcal{D}_h[G]$ of $\mathfrak{M}_h[G]$ and use $\mathcal{D}_h[G]$ to study the structure of $\mathfrak{M}_h[G]$. 
\end{abstract}

\section{Introduction} 


The moduli theory of algebraic varieties was motivated by the attempt to fully understand Riemann's assertion in \cite{Rie57} that the isomorphism classes of Riemann surfaces of genus $g>1$ depend on $(3g-3)$ parameters (called "moduli"). The modern approach to moduli problems via functors was developed by Grothendieck and Mumford (cf. \cite{MF82}), and later by Gieseker, Koll{\'a}r, Viehweg, et al (cf. \cite{Gie77}, \cite{Kol13}, \cite{Vie95}). The idea is to define a moduli functor for the given moduli problem and study the representability of the moduli functor via an algebraic variety or some other geometric object. For instance, in the case of smooth projective curves of genus $g\geq 2$, we consider the (contravariant) functor $\mathcal{M}_g$ from the category of schemes to the category of sets, such that 
\begin{itemize}
\item[(1)] For any scheme $T$, $\mathcal{M}_g(T)$ consists of the $T$-isomorphism classes of flat projective families of curves of genus $g$ over the base $T$.
\item[(2)] Given a morphism $f:S\rightarrow T$, $\mathcal{M}_g(f):\mathcal{M}_g(T)\rightarrow \mathcal{M}_g(S)$ is the map associated to the pull back.
\end{itemize}
It has been shown by Mumford that there exists a quasi-projective {\em coarse moduli scheme} $\mathfrak{M}_g$ for the functor $\mathcal{M}_g$ (cf. \cite{Mum62}), in the following sense:\\
there exists a natural transformation $\eta:\mathcal{M}_g\rightarrow \Hom(-,\mathfrak{M}_g)$, such that $$\eta_{\mathbf{Spec}(\mathbb{C})}:\mathcal{M}_g(\mathbf{Spec}(\mathbb{C}))\rightarrow \Hom(\mathbf{Spec}(\mathbb{C}),\mathfrak{M}_g)$$ is bijective and $\eta$ is universal among such natural transformations. This means that the closed points of $\mathfrak{M}_g$ are in one to one correspondence with the isomorphism classes of curves of genus $g$ and given a family $\mathfrak{X}\rightarrow T$ of curves of genus $g$, we have a morphism (induced by $\eta$) from $T$ to $\mathfrak{M}_g$ such that any (closed) point $t\in T$ is mapped to $[\mathfrak{X}_t]$ in $\mathfrak{M}_g$. \\
The definitions are the same in higher dimensions, if one replaces curves of genus $g\geq 2$ by Gorenstein varieties with ample canonical classes. The existence of a coarse moduli space is then more difficult to prove, we refer to \cite{Vie95} and \cite{Kol13} for further discussions.\\
For several purposes, it is important to generalize the method to moduli problems of varieties which admit an effective action by a given finite group $G$.
 Here we consider the concept of a {\em $G$-marked variety}, which is a triple ($X$,$G$,$\alpha$) such that $X$ is a projective variety and $\alpha:G\times X\rightarrow X$ is a faithful action. The isomorphisms between $G$-marked varieties are $G$-equivariant isomorphisms (for more details, see definition \ref{def1}). In similarity to the case of $\mathcal{M}_g$, we study in this article the moduli functor $\mathsf{M}^G_h$ of $G$-marked Gorenstein canonical models with Hilbert polynomial $h$ such that, for any scheme $T$, $\mathsf{M}^G_h(T)$ is the set of $T$-isomorphism classes of $G$-marked flat families of Gorenstein canonical models with Hilbert polynomial $h$ over the base $T$, and given a morphism $f:S\rightarrow T$, $\mathsf{M}^G_h(f)$ is the map associated to the set of pull-backs (cf. \ref{def3}). \\
We refer to the recently published survey article \cite{Cat15}, Section 10, for some applications in the case of algebraic curves and surfaces; there the author discusses several topics on the theory of $G$-marked curves  and sketches the construction of the moduli space of $G$-marked canonical models of  surfaces.\\
The main theorem of this article is the following:
\begin{theorem}\label{thm}
Given a finite group $G$ and a Hilbert polynomial $h\in \mathbb{Q}[t]$, there exists a quasi-projective coarse moduli scheme $\mathfrak{M}_h[G]$ for $\mathsf{M}^G_h$, the moduli functor of $G$-marked Gorenstein canonical models with Hilbert polynomial $h$.
\end{theorem}
The above result, while being a consequence of hard results of Viehweg and Koll{\'a}r, fills a gap in the existing literature and shall hopefully be useful for future investigations.\\
 
The structure of this paper is as follows:\\
In section 2 we introduce the definition of "$G$-marked varieties" and the associated moduli problem by defining the moduli functor $\mathsf{M}^G_h$ for a given group $G$ and Hilbert polynomial $h$. \\
In section 3 we study two basic properties (boundedness and local closedness) of the moduli functor $\mathsf{M}^G_h$. \\
Recall that a moduli functor of varieties $\mathsf{M}$ is called {\em bounded} if the objects in $\mathsf{M}({\bf Spec}(\mathbb{C}))$ are parameterized by a finite number of families (cf. \ref{bounded}). 
In this article we show in (\ref{cor1}) that $\mathsf{M}^G_h$ is bounded by a family $U_{N,h'}^G\rightarrow H^{G}_{N,h'}$ over an appropriate subscheme of a Hilbert scheme. \\
However the family $U_{N,h'}^G\rightarrow H^{G}_{N,h'}$ that we get in (\ref{cor1}) may not belong to $\mathsf{M}^G_h(H^{G}_{N,h'})$, i.e., not every fibre of the family is a $G$-marked canonical model. Here comes the problem of local closedness: roughly speaking, a moduli functor $\mathsf{M}$ of varieties is called {\em locally closed} if for any  flat projective family $\mathfrak{X}\rightarrow T$, the subset $\{t\in T|[\mathfrak{X}_t]\in \mathsf{M}(\mathbf{Spec}(\mathbb{C}))\}$ is locally closed in $T$ (see \ref{localclose} for more details). We solve this problem in (\ref{prop5}) by taking a locally closed subscheme $\bar{H}^G_{N,h'}$ of $H^{G}_{N,h'}$ and considering the restriction of $U_{N,h'}^G\rightarrow H^{G}_{N,h'}$ to $\bar{H}^G_{N,h'}$.\\
In Section 4 we first apply Geometric Invariant Theory, obtaining the quotient $\mathfrak{M}_h[G]$ of $\bar{H}^G_{N,h'}$ by some reductive groups. Then we prove that $\mathfrak{M}_h[G]$ is the coarse moduli scheme for our moduli functor $\mathsf{M}^G_h$.\\
In Section 5 we obtain the canonical representation type decomposition $\mathcal{D}_h[G]$ of our moduli space $\mathfrak{M}_h[G]$. We apply different methods to study $\mathcal{D}_h[G]$ in the general setting as well as in the case of algebraic curves.
\begin{remark}
A referee suggested an alternative proof of theorem \ref{thm}, obtained via the idea of constructing a suitable substack of a fibred product of the inertia stack and using then the deep result of Keel and Mori (cf. \cite{KM97}).\\
However we believe that our simple proof relying on GIT techniques sheds light upon a canonical representation type decomposition $\mathcal{D}_h[G]$ (cf.\ref{CanDecomp}) of the moduli space $\mathfrak{M}_h[G]$, which will be useful to study the structure of $\mathfrak{M}_h[G]$. 
\end{remark}
\section{G-marked varieties}
In this article we work over the complex field $\mathbb{C}$. By a "scheme" we mean a separated scheme of finite type over $\mathbb{C}$, a point in a scheme is assumed to be a closed point. Moreover, $G$ shall always denote a finite group.
\begin{definition}[\cite{Cat15}, definition 181]\label{def1}
(1) A \textit{G-marked (projective) variety (resp. scheme)} is a triple $(X,G,\rho)$ where $X$ is a projective variety (resp. scheme) and $\rho:G\rightarrow \Aut(X)$ is an injective homomorphism. Or equivalently, it is a triple $(X,G,\alpha)$ where $\alpha:X\times G\rightarrow X$ is a faithful action of $G$ on $X$.\\
(2) A morphism $f$ between two G-marked varieties $(X,G,\rho)$ and $(X',G,\rho ')$ is a $G$-equivariant morphism $f:X\rightarrow X'$, i.e.,
$\forall g\in G, f\circ \rho(g) = \rho'(g)\circ f$.\\
(3) A family of $G$-marked varieties (resp. schemes) is a triple $((p:\mathfrak{X}\rightarrow T),G,\rho)$, where $G$ acts faithfully on $\mathfrak{X}$ via an injective homomorphism $\rho:G\rightarrow \Aut(\mathfrak{X})$ and trivially on $T$; $p$ is flat, projective, $G$-equivariant and $\forall t\in T$, the induced triple $(\mathfrak{X}_t,G,\rho_t)$ is a $G$-marked variety (resp. scheme).\\
(4) A morphism between two $G$-marked families $((p:\mathfrak{X}\rightarrow T),G,\rho)$ and $((p':\mathfrak{X'}\rightarrow T'),G,\rho')$ 
is a commutative diagram:
$$\begin{CD}
\mathfrak{X} @>\tilde{f}>>   \mathfrak{X}'\\ 
@VVpV                           @VVp'V\\
T                @>f>>             T'  
\end{CD}$$
where $\tilde{f}:\mathfrak{X}\rightarrow \mathfrak{X}'$ is a $G$-equivariant morphism.\\
(5) Let $((p:\mathfrak{X}\rightarrow T),G,\rho)$ be a $G$-marked family and $f:S\rightarrow T$ a morphism. Denoting by $\mathfrak{X}_S$ (or $f^*\mathfrak{X}$) the fiber product of $f$ and $p$, $\rho$ induces a $G$-action $\rho _S$ (or $f^*\rho$) on $\mathfrak{X}_S$ such that $((p_S:\mathfrak{X}_S\rightarrow S),G,\rho _S)=:f^*((p:\mathfrak{X}\rightarrow T),G,\rho)$ is 
again a $G$-marked family.
\end{definition}

\begin{definition}\label{def2}
A normal projective variety $X$ is called a \textit{canonical model} if $X$ has canonical singularities (cf. \cite{Rei87}) and $K_X$ is ample.
\end{definition}

\begin{definition}\label{def3}
Denote by $\mathfrak{Sch}$ the category of schemes (over $\mathbb{C}$). The \textit{moduli functor of $G$-marked Gorenstein canonical models} with Hilbert polynomial $h\in \mathbb{Q}[t]$ is a contravariant functor:
$$\mathsf{M}^G_h:\mathfrak{Sch}\rightarrow \mathfrak{Sets},\  such\ that$$ 
(1) For any scheme $T$,
\begin{align*}
\mathsf{M}^G_h(T):=\{((p:\mathfrak{X}\rightarrow T),G,\rho)|\ p \text{ is flat and projective, all fibres of p}\\
 \text{ are canonical models, } \omega_{\mathfrak{X}/T} \text{ is invertible, } \\
        \forall t \in T,\forall k\in \mathbb{N}, \chi(\mathfrak{X}_t,\omega^k_{\mathfrak{X}_t})=h(k)\}/\simeq 
 \end{align*}
where "$\simeq$" is the equivalence relation given by the isomorphisms of $G$-marked families over $T$ (i.e., in the commutative diagram of \ref{def1} (4), take $T'=T$ and $f=id_T$).\\
(2) Given $f\in \Hom(S,T)$, $\mathsf{M}^G_h (f): \mathsf{M}^G_h(T) \rightarrow \mathsf{M}^G_h(S)$ is the map associated to the pull back, i.e.,
$$[((p:\mathfrak{X}\rightarrow T),G,\rho)] \mapsto [((p_S:\mathfrak{X}_S\rightarrow S),G,\rho_S)].$$
\end{definition}
\begin{remark}
In this article, whenever we write $((\mathfrak{X}\rightarrow T),G,\rho)\in \mathsf{M}^G_h(T)$, we mean choosing a representative $((\mathfrak{X}\rightarrow T),G,\rho)$ from the isomorphism class $[((\mathfrak{X}\rightarrow T),G,\rho)]\in \mathsf{M}^G_h(T)$.
\end{remark}
In the case where $G$ is trivial, we denote by $\mathsf{M}_h$ the corresponding functor.  \\

\section{Basic properties of $\mathsf{M}^G_h$}
In this section we study two important properties of the moduli functor $\mathsf{M}^G_h$: boundedness and local closedness. The main results are (\ref{prop4}), (\ref{cor1}) for boundedness and (\ref{prop5}) for local closedness.

\begin{definition}\label{bounded}
A moduli functor $\mathsf{M}$ of varieties is called {\em bounded} if there exists a flat and projective family $\mathfrak{U}\rightarrow S$ over a scheme $S$ such that $\forall [X]\in \mathsf{M}({\bf Spec}(\mathbb{C}))$, $X$ is isomorphic to a fibre $\mathfrak{U}_s$ for some $s\in S$. (For a stronger definition, see \cite{Kov09}, Definition 5.1)
\end{definition}

In the case where $G$ is trivial boundedness is already known (cf. \cite{Kar00}, \cite{Mat86}). However we can not apply it directly to the general case since we have an action by $G$.
Here we introduce the notion of "bundle of $G$-frames" to solve this problem.\\
Let $Y$ be a scheme and $\mathcal{E}$ a locally free sheaf of rank $n$ on $Y$. Set 
$$\mathbb{V}(\mathcal{E}):=\mathbf{Spec}_Y Sym(\mathcal{E}^{\vee}),$$
the geometric vector bundle associated to $\mathcal{E}$ over $Y$ (cf. \cite{Har77}\footnote{The bundle $\mathbb{V}(\mathcal{E})$ defined here is in fact dual to that of \cite{Har77}.}, Exercise II.5.18). 
\begin{definition}[Frame Bundle]\label{def4}
Let $\mathcal{E}$ be a locally free sheaf of rank $n$ on a scheme $Y$. In this article we call (what is in bundle theory called) the principal bundle associated to $\mathbb{V}(\mathcal{E})$ the \textit{frame bundle} $\mathcal{F}(\mathcal{E})$ of $\mathcal{E}$ over $Y$. For any $y\in Y$, the fibre $\mathcal{F}(\mathcal{E})_y$ over $y$ is called the set of frames (i.e., bases) for the vector space $\mathcal{E}\otimes \mathbb{C}(y)$. \\
Hence $\mathcal{F}(\mathcal{E})$ is the open subscheme  of $\mathbb{V}(\mathcal{H}om_{\mathcal{O}_Y}(\mathcal{O}^n_Y, \mathcal{E}))$ such that $\forall y \in Y$, the fibre $\mathcal{F}(\mathcal{E})_y$ corresponds to the invertible homomorphisms. We denote a point in $\mathcal{F}(\mathcal{E})$ as a pair $(y,\psi)$, where $y$ is a point in $Y$ and $\psi:\mathbb{C}^n\rightarrow \mathcal{E}\otimes \mathbb{C}(y)$ is an isomorphism of $\mathbb{C}$-vector spaces.
\end{definition}

\begin{prop}\label{prop1}
Let $\mathcal{E}$ be a locally free sheaf of rank $n$ on a scheme $Y$ and $p:\mathcal{F}(\mathcal{E})\rightarrow Y$ the natural projection.
There exists a tautological isomorphism $\phi_{\mathcal{E}}:\mathcal{O}^n_{\mathcal{F}(\mathcal{E})}\rightarrow p^*\mathcal{E}$ of sheaves on $\mathcal{F}(\mathcal{E})$ such that for any point $z:=(y,\psi)\in \mathcal{F}(\mathcal{E})$, $\phi_{\mathcal{E}}|_{\{z\}}=\psi$ via the isomorphism $Hom(\mathbb{C}^n,p^*\mathcal{E}\otimes \mathbb{C}(z))\simeq Hom(\mathbb{C}^n,\mathcal{E}\otimes \mathbb{C}(y))$.
\end{prop}
\begin{proof}
This proposition is well known (the idea is similar to that of \cite{Gro58}). Observe that $p^*\mathcal{E}$ has $n$ global sections $s_1(\mathcal{E}),...,s_n(\mathcal{E})$
 such that for any $z=(y,\psi)\in \mathcal{F}(\mathcal{E})$, $s_i(\mathcal{E})\otimes \mathbb{C}(z)=\psi(e_i)$, where $\{e_i\}^n_{i=1}$ is the canonical basis of $\mathbb{C}^n$ and we identify $p^*\mathcal{E}\otimes \mathbb{C}(z)$ with $\mathcal{E}\otimes \mathbb{C}(y)$. Then the universal basis morphism $\phi_\mathcal{E}:=(s_1(\mathcal{E}),...,s_n(\mathcal{E})):\mathcal{O}^n_{\mathcal{F}(\mathcal{E})}\rightarrow p^*\mathcal{E}$ is an isomorphism of locally free sheaves.
\end{proof}
\begin{remark}\label{Can}
The set of sections $\{s_i(\mathcal{E})\}^n_{i=1}$ (or equivalently, the isomorphism $\phi_{\mathcal{E}}$) satisfies the following compatibility conditions:\\
(1) Let $f:X\rightarrow Y$ be a morphism and  $f_\mathcal{F}:\mathcal{F}(f^*\mathcal{E})\rightarrow \mathcal{F}(\mathcal{E})$ the induced morphism: we have that $f^*_{\mathcal{F}}(s_i(\mathcal{E}))=s_i(f^*\mathcal{E})$.\\
(2) Given an isomorphism $l:\mathcal{E}_1\rightarrow \mathcal{E}_2$ of locally free sheaves on $Y$, noting that the induced isomorphism $l_{\mathcal{F}}:\mathcal{F}(\mathcal{E}_1)\rightarrow \mathcal{F}(\mathcal{E}_2)$ commutes with the projections $p_j:\mathcal{F}(\mathcal{E}_j)\rightarrow Y$, $j=1,2$, we have then $l^*_{\mathcal{F}}(\phi_{\mathcal{E}_2})=p_1^*(l)\circ \phi_{\mathcal{E}_1}$.
\end{remark}
\begin{definition}\label{def5}
Let $\mathcal{E}$ be a locally free sheaf of rank $n$ on a scheme $Y$: we say that a group $G$ acts \textit{faithfully and linearly} on $\mathcal{E}$ if \\
(1) the action is given by an injective homomorphism $\rho:G\hookrightarrow \Aut_{\mathcal{O}_Y}(\mathcal{E})$;\\
(2) $\forall y\in Y$, the induced action $\rho_y$ is a faithful $G$-representation on $\mathbb{C}^n$.\\
In this case we call the pair $(\mathcal{E},\rho)$ a {\em locally free $G$-sheaf }.
\end{definition}

\begin{definition}\label{def6}
(1) Given $\phi \in \Aut(Y)$, let $\Gamma_{\phi}:Y\rightarrow Y\times Y$ be the graph map of $\phi$. The \textit{fixpoints scheme} of $\phi$ (denoted by Fix($\phi$)) is the (scheme-theoretic) inverse image of $\Delta$ by $\Gamma_{\phi}$, where $\Delta$ is the diagonal subscheme of $Y\times Y$.\\
(2) Given a $G$-action on $Y$, the \textit{fixpoints scheme} of $G$ on $Y$ is: 
$$Y^G:=\cap_{g\in G}Fix (\phi_g),$$ where $\phi_g:Y\rightarrow Y, y\mapsto gy$. 
\end{definition}
\begin{remark}\label{rk2}
(1) Let $f:X\rightarrow Y$ be a $G$-equivariant morphism between two schemes on which $G$ acts: we have a natural restriction morphism $f|_{X^G}:X^G\rightarrow Y^G$.\\
(2) Given a $G$-action on $Y$ and a subgroup $H$ of $G$, there is an induced $C(H)$-action on $Y^H$, where $C(H)$ denotes the centralizer group of $H$ in $G$.
\end{remark}
\begin{definition}\label{def7}
Let $(\mathcal{E},\rho)$ be a locally free $G$-sheaf of rank $n$ on $Y$. Given a faithful linear representation $\beta: G\rightarrow GL(n,\mathbb{C})$, we define an action $(\beta,\rho)$ of $G$ on $\mathcal{H}om_{\mathcal{O}_Y}(\mathcal{O}_Y^n,\mathcal{E})$: $\forall g\in G$, open subset $U\subset Y, \phi \in \mathcal{H}om_{\mathcal{O}_Y}(\mathcal{O}_Y^n,\mathcal{E})(U)\ and\ s\in \mathcal{O}_Y^n(U)$; $(g\phi)(s):=\rho(g)\phi(\beta(g^{-1})s)$. The action $(\beta,\rho)$ restricts naturally to $\mathcal{F}(\mathcal{E})$, we denote by $\mathcal{F}(\mathcal{E},G,\rho;\beta)$ the corresponding fixpoints scheme: it is called the \textit{bundle of $G$-frames of $\mathcal{E}$ associated to the action $\rho$ with respect to $\beta$}.
\end{definition} 
\begin{remark}\label{rk2.5}
(1) Denoting by $C(G,\beta)$ the centralizer group of $\beta(G)$ in $GL(n,\mathbb{C})$, an easy observation is that $\forall y\in Y$, the fiber $\mathcal{F}(\mathcal{E},G,\rho;\beta)_y$ corresponds to the set of $G$-equivariant isomorphisms between the $G$-linear representations $\beta$ and $\rho_y$. Therefore we have that either $\mathcal{F}(\mathcal{E},G,\rho;\beta)_y=\emptyset$,  or $\mathcal{F}(\mathcal{E},G,\rho;\beta)_y\simeq C(G,\beta)$.\\
(2) If $\beta,\beta':G\rightarrow GL(n,\mathbb{C})$ are equivalent representations (i.e., there exists $g\in GL(n,\mathbb{C})$ such that $\beta'=g\beta g^{-1}$), then we have that $\mathcal{F}(\mathcal{E},G,\rho;\beta)\simeq \mathcal{F}(\mathcal{E},G,\rho;\beta')$.
\end{remark}
Observe that if $Y$ is connected and there exists $y\in Y$ such that $\mathcal{F}(\mathcal{E},G,\rho;\beta)_y\simeq C(G,\beta)$, then $\mathcal{F}(\mathcal{E},G,\rho;\beta)_{y'}\simeq C(G,\beta)$ for all $y'\in Y$ (See \cite{Cat13}, Prop 37), hence we have the following definition:
\begin{definition}
Let $Y$ be a connected scheme and $(\mathcal{E},\rho)$ a locally free $G$-sheaf of rank $n$ on $Y$. We say that ($\mathcal{E},\rho)$ (or $\mathcal{E}$ if the action is clear from the context) has \textit{decomposition type $\beta$}, where $\beta: G\rightarrow GL(n,\mathbb{C})$ is a faithful representation, if there exists $y\in Y$, such that $\rho_y\simeq \beta$.
\end{definition}

\begin{definition}[Bundle of $G$-frames]\label{def8}
Let ($\mathcal{E},\rho)$ be a locally free $G$-sheaf of rank $n$ on a scheme $Y$. We define the \textit{bundle of $G$-frames of $\mathcal{E}$ associated to $\rho$}, denoted by $\mathcal{F}(\mathcal{E},G,\rho)$ (or $\mathcal{F}(\mathcal{E},G)$ when $\rho$ is clear from the context), as follows:\\
If $Y$ is connected and $\mathcal{E}$ has decomposition type $\beta$, then $\mathcal{F}(\mathcal{E},G,\rho):=\mathcal{F}(\mathcal{E},G,\rho;\beta)$.\\
In general we decompose $Y$ into the union of connected components $Y=\sqcup Y_i$ and $\mathcal{F}(\mathcal{E},G,\rho)$ is the (disjoint) union of all the $\mathcal{F}(\mathcal{E}|_{Y_i},G,\rho|_{Y_i})$.
\end{definition}
\begin{remark}
Since we can vary $\beta$ in its equivalence class,  we see from (\ref{rk2.5}) that $\mathcal{F}(\mathcal{E},G)$ is unique up to isomorphisms. 
\end{remark}
\begin{definition}\label{def9}
Let $(\mathcal{E},\rho)$ be a {\em free} $G$-sheaf of rank $n$ on a scheme $Y$: the action is said to be \textit{defined over $\mathbb{C}$} if there exists a $G$-equivariant isomorphism $\phi: (\mathcal{O}^n_Y,\beta)\rightarrow (\mathcal{E},\rho)$, where $\beta$ is a faithful $G$-representation on $\mathbb{C}^n$.
\end{definition}
\begin{prop}\label{prop2}
Given $(\mathcal{E},\rho)$ a locally free $G$-sheaf of rank $n$ on a connected scheme $Y$ with decomposition type $\beta$, the projection $p:\mathcal{F}(\mathcal{E},G)\rightarrow Y$ induces an action $p^*\rho$ on $p^*\mathcal{E}$. Then $(p^*\mathcal{E},p^*\rho)$ is defined over $\mathbb{C}:$ the morphism $\phi_{\mathcal{E},G}:=\phi_{\mathcal{E}}|_{\mathcal{F}(\mathcal{E},G)}:(\mathcal{O}^n_{\mathcal{F}(\mathcal{E},G)},\beta)\rightarrow (p^*\mathcal{E},p^*\rho)$ is a $G$-equivariant isomorphism, where $\phi_\mathcal{E}$ is the universal basis morphism defined in $(\ref{prop1})$.
\end{prop}
\begin{proof}
It is clear that $\phi_{\mathcal{E},G}$ is an isomorphism of sheaves, what remains to show is that $\phi_{\mathcal{E},G}$ is $G$-equivariant. Since $\phi_{\mathcal{E},G}$ is an isomorphism of locally free sheaves, it suffices to show that $\forall (y,\psi)\in \mathcal{F}(\mathcal{E},G)$, $\phi_{\mathcal{E},G}|_{\{(y,\psi)\}}$ is $G$-equivariant. By our construction in (\ref{prop1}), we have that $p^{-1}(y)\subset \mathbb{V}(\mathcal{H}om_{\mathcal{O}_Y}(\mathcal{O}_Y^n,\mathcal{E}))^G_{y} \simeq \Hom(\mathbb{C}^n,\mathcal{E}\otimes \mathbb{C}(y))^G$, where the $G$-action is $(\beta,\rho_y)$. Under this isomorphism the point $(y,\psi)$ corresponds exactly to $\phi_{\mathcal{E},G}|_{\{(y,\psi)\}}$, hence $\phi_{\mathcal{E},G}|_{\{(y,\psi)\}}$ is $G$-equivariant.
\end{proof} 
\begin{remark}\label{CanG}
Given a locally free $G$-sheaf $(\mathcal{E},\rho)$ of rank $n$ on $Y$, setting $s_i(\mathcal{E},G):=s_i(\mathcal{E})|_{\mathcal{F}(\mathcal{E},G)}$, then $\{s_i(\mathcal{E},G)\}$ and $\phi_{\mathcal{E},G}$ have similar properties as $\{s_i(\mathcal{E})\}$ and $\phi_{\mathcal{E}}$ have in (\ref{Can}).
\end{remark}
\begin{prop}\label{prop3}
Assume that $Y$ is connected and $(\mathcal{E},\rho)$ is a locally free $G$-sheaf of rank $n$ on Y with decomposition type $\beta$. Then there is a natural $C(G,\beta)$-action on $\mathcal{F}(\mathcal{E},G)$ and $Y$ is a categorical quotient of $\mathcal{F}(\mathcal{E},G)$ by $C(G,\beta)$.
\end{prop}
\begin{proof}
To see the $C(G,\beta)$-action, it suffices to notice that the actions $\beta$ and $\rho$ on $\mathcal{F}(\mathcal{E})$ commute, i.e., $\forall g\in G$, $\beta(g)\rho(g)=\rho(g)\beta(g)$ as elements in $\Aut(\mathcal{F}(\mathcal{E}))$. \\
From the definition of $\mathcal{F}(\mathcal{E},G)$ (cf. \ref{def8}), one observes that the projection $p:\mathcal{F}(\mathcal{E},G)\rightarrow Y$ is affine and $C(G,\beta)$-equivariant, therefore we may assume that $Y,\mathcal{F}(\mathcal{E},G)$ are affine schemes and $A$ (resp. $B$) is the coordinate ring of $Y$ (resp. $\mathcal{F}(\mathcal{E},G)$). Since $p$ is surjective and $C(G,\beta)$-equivariant, we have that $A\subset B^{C(G,\beta)}\subset B$. Noting that $B$ is a finitely generated $\mathbb{C}$-algebra and $C(G,\beta)$ is a reductive group (cf. \ref{rk3}), we conclude that $B^{C(G,\beta)}$ is a finitely generated $\mathbb{C}$-algebra and $\mathbf{Spec}B^{C(G,\beta)}$ is the universal categorical quotient of $\mathcal{F}(\mathcal{E},G)$ by $C(G,\beta)$ (cf. \cite{MF82}, p.27). Now since every fibre of $p$ is a  closed $C(G,\beta)$-orbit (in fact isomorphic to $C(G,\beta)$), which must be mapped to a point in $\mathbf{Spec} B^{C(G,\beta)}$, for dimensional reasons we conclude that $B^{C(G,\beta)}$ is a finite $A$-module. For any maximal ideal $m$ of $A$, by the property of a universal categorical quotient (cf. \cite{MF82}, p.4) we see that $\mathbf{Spec}(B^{C(G,\beta)}\otimes_A \mathbb{C}(m))$ is the categorical quotient of $p^{-1}(\mathbf{Spec}(\mathbb{C}(m)))\simeq C(G,\beta)$ by $C(G,\beta)$, hence $B^{C(G,\beta)}\otimes_A \mathbb{C}(m)=\mathbb{C}$, which implies that $(B^{C(G,\beta)}/A)\otimes_A \mathbb{C}(m)=0$. By Nakayama's lemma, we have that $(B^{C(G,\beta)}/A)_m=0$, which implies that $A=B^{C(G,\beta)}$.
\end{proof}
 Before stating the Boundedness theorem, let us first recall the action of general linear groups on Hilbert schemes (cf. \cite{Vie95}, Section 7.1).\\
 Denote by $H_{n,h}$ the Hilbert scheme of closed subschemes of $\mathbb{P}^n$ with Hilbert polynomial $h$ and by $U_{n,h}\subset H_{n,h}\times \mathbb{P}^n$ the universal family. Let $\Phi:GL(n+1,\mathbb{C})\times \mathbb{P}^n\rightarrow \mathbb{P}^n$ be the natural action, so that there is an action $\Psi:GL(n+1)\times H_{n,h}\rightarrow H_{n,h}$ such that $\forall g\in GL(n+1,\mathbb{C})$, $U_{n,h}$ is invariant under the morphism $\Psi_g\times \Phi_{g}$. 
 
 
Given a (finite) group $G$, a faithful representation of $G$ on $V:=\mathbb{C}^{n+1}$ is given by an injective homomorphism $\beta:G\rightarrow GL(n+1,\mathbb{C})$, or equivalently, by a decomposition $V=\bigoplus _{\rho\in Irr(G)}W^{n(\rho)}_\rho$. Two representations are equivalent (i.e. the images of $G$ are conjugate as subgroups of $GL(n+1,\mathbb{C})$) if and only if they have the same decomposition type (cf. \cite{Ser77}, Chap.2), hence the set of  equivalence classes $\mathcal{B}_n$ of $G$-representations on $V$ is finite.
\begin{definition}\label{def10}
Given $\beta:G\rightarrow GL(n+1,\mathbb{C})$ a faithful representation, it induces an action $\Psi|_{\beta(G)}$ of $G$ on $H_{n,h}$. Define $H^{G,\beta}_{n,h}$ as the fixpoints scheme of the $\beta(G)$-action on $H_{n,h}$ and denote by $U_{n,h}^{G,\beta}$ the restriction of $U_{n,h}$ from $H_{n,h}$ to  $H^{G,\beta}_{n,h}$.
\end{definition}
\begin{remark}\label{rk3}
(1) We have already seen that $C(G,\beta)$, the centralizer group of $\beta(G)$ in $GL(n+1,\mathbb{C})$, has a natural action on $H^{G,\beta}_{n,h}$ (cf. \ref{rk2}). By Schur's Lemma one obtains that $C(G,\beta)\simeq \Pi_{\rho\in Irr(G)}GL(n(\rho),\mathbb{C})$, hence $C(G,\beta)$ is reductive. \\
(2) Let $\beta,\beta'$ be two equivalent representations, such that $\beta'=g\beta g^{-1}$ for some $g\in GL(n+1,\mathbb{C})$. We have that $H^{G,\beta}_{n,h}$ is isomorphic to $H^{G,\beta'}_{n,h}$ via $\Psi_g$ as subschemes of $H_{n,h}$.\\
(3) Since $U_{n,h}^{G,\beta}$ (as a subscheme of $H^{G,\beta}_{n,h}\times \mathbb{P}^n$) is invariant under the action $id\times (\Phi|_{\beta(G)})$, we obtain a $G$-marked family $((p_\beta:U_{n,h}^{G,\beta}\rightarrow H^{G,\beta}_{n,h}),G,\beta)$.
\end{remark}
\begin{definition}\label{def11}
Let $V$ be a $\mathbb{C}$-vector space of dimension $n+1$. Denoting by $\mathcal{B}_n$ the set of equivalence classes of linear representations of $G$ on $V$, we pick one representative in each equivalence class of $\mathcal{B}_n$ and define: $$((p:U_{n,h}^G\rightarrow H^{G}_{n,h}),G,\mathcal{B}_n):=\bigsqcup_{[\beta]\in \mathcal{B}_n}((p_\beta:U_{n,h}^{G,\beta}\rightarrow H^{G,\beta}_{n,h}),G,\beta),$$
where "$\bigsqcup$" means a disjoint union.\\
 Note that two different choices of representatives result in isomorphic families.
\end{definition}
By Matsusaka's big theorem (\cite{Mat86}, Theorem 2.4), there exists an integer $k_0$ such that $\forall [X]\in \mathsf{M}_h(\mathbf{Spec}\mathbb{C})$, $\omega^{k_0}_X$ is very ample and has vanishing higher cohomology groups, we fix one such $k_0$ for the rest of this article (we refer to \cite{Siu93}, \cite{Dem96} and \cite{Siu02} for effective bounds on $k_0$). Given a family ($p:\mathfrak{X}\rightarrow T$) $\in \mathsf{M}_h(T)$, by "Cohomology and Base change" (cf. \cite{Mum70},  II.5), $p_{*}(\omega^{k{_0}}_{\mathfrak{X}/T})$ is a locally free sheaf of rank $h(k_0)$. Moreover we have a surjection $p^*p_{*}(\omega^{k{_0}}_{\mathfrak{X}/T})\twoheadrightarrow \omega^{k{_0}}_{\mathfrak{X}/T}$, which induces a $T$-embedding $i:\mathfrak{X}\hookrightarrow \mathbb{P}(p_{*}(\omega^{k{_0}}_{\mathfrak{X}/T}))$ such that $\omega^{k{_0}}_{\mathfrak{X}/T}\simeq i^*(\mathcal{O}_{\mathbb{P}(p_{*}(\omega^{k{_0}}_{\mathfrak{X}/T}))}(1))$ (cf. \cite{Har77}, II.7.12).  Assuming in addition that $p_{*}(\omega^{k{_0}}_{\mathfrak{X}/T})$ is trivial, the $T$-embedding becomes $i:\mathfrak{X}\hookrightarrow T\times \mathbb{P}^N$ ($N:=h(k_0)-1$). Setting $h'(k):=h(k_0k)$, there exists a morphism $f:T\rightarrow H_{N,h'}$ such that $\mathfrak{X}\simeq f^*U_{N,h'}$. Now  taking the group action into account, we have the following:
\begin{prop}[Boundedness]\label{prop4}
Given $((p:\mathfrak{X}\rightarrow T),G,\rho)\in \mathsf{M}^G_h(T)$, we denote by $\bar{\rho}$ the induced action of $G$ on $p_{*}(\omega^{k{_0}}_{\mathfrak{X}/T})$. Assume that $p_{*}(\omega^{k{_0}}_{\mathfrak{X}/T})$ is trivial and $\bar{\rho}$ is defined over $\mathbb{C}$, then there exists $f: T\rightarrow H^G_{N,h'}$ such that, $((\mathfrak{X} \rightarrow T),G,\rho)\simeq  f^*((U_{N,h'}^G\rightarrow H^{G}_{N,h'}),G,\mathcal{B}_N)$ and $\omega^{k_0}_{\mathfrak{X}/T}\simeq \mathcal{O}_{T\times \mathbb{P}^N}(1)|_{\mathfrak{X}}$.
\end{prop}
\begin{proof}
It suffices to prove the statement on each connected component of $T$, hence we may assume that $T$ is connected and $p_{*}(\omega^{k{_0}}_{\mathfrak{X}/T})$ has decomposition type $\beta$. \\
The action $\bar{\rho}$ induces an action of $G$ on $\mathbf{Proj}_{T}(p_{*}(\omega^{k{_0}}_{\mathfrak{X}/T}))=T\times \mathbb{P}^N$ such that the embedding $i:\mathfrak{X}\rightarrow T\times \mathbb{P}^N$ is $G$-equivariant. Since by assumption $\bar{\rho}$ is defined over $\mathbb{C}$, we may require that the $G$-action on $T\times \mathbb{P}^N$ is given by $\pi_2^*(\beta)$, where $\pi_2:T\times \mathbb{P}^N\rightarrow \mathbb{P}^N$ is the projection onto the second factor. Now by the universal property of the Hilbert scheme, there exists $f:T\rightarrow H_{N,h'}$, such that $i(\mathfrak{X})= (f\times id_{\mathbb{P}^N})^*U_{N,h'}$. To complete the proof, it remains to show that $f$ factors through $H^{G,\beta}_{N,h'}$, which is equivalent to the property that $\forall g\in G,\Psi_{\beta(g)}\circ f=f$; again by the universal property of the Hilbert scheme, this is equivalent to showing that $\forall g\in G$, $((\Psi_{\beta(g)}\circ f)\times id_{\mathbb{P}^N})^*U_{N,h'}=i(\mathfrak{X})$. However we have that 
$$( (\Psi_{\beta(g)}\circ f)\times id_{\mathbb{P}^N})^*U_{N,h'}=(f\times id_{\mathbb{P}^N})^*(\Psi_{\beta(g)}\times  id_{\mathbb{P}^N})^*U_{N,h'}$$
$$=(f\times id_{\mathbb{P}^N})^*(id_{U_{N,h'}}\times \Phi_{\beta(g)^{-1}})^*U_{N,h'}=(id_T\times \Phi_{\beta(g)^{-1}})^*(f\times id_{\mathbb{P}^N})^*U_{N,h'}$$
$$=(id_T\times \Phi_{\beta(g)^{-1}})^*(i(\mathfrak{X})),$$ 
which is simply $i(\mathfrak{X})$ as the embedding $i:\mathfrak{X}\rightarrow T\times\mathbb{P}^N$ is $G$-equivariant.
\end{proof}
Combining (\ref{prop2}) with (\ref{prop4}), we have the following corollary:
\begin{corollary}\label{cor1}
For any scheme $T$ and $ ((p:\mathfrak{X}\rightarrow T),G,\rho)\in \mathsf{M}^G_{h}(T)$, let $q:\mathcal{F}(p_*(\omega^{k_0}_{\mathfrak{X}/T}),G)\rightarrow T$ be the bundle of $G$-frames of $p_*(\omega^{k_0}_{\mathfrak{X}/T})$ over T. Then the isomorphism $\phi_{p_*(\omega^{k_0}_{\mathfrak{X}/T}),G}$ induces a morphism $$f_{\mathfrak{X}/T,k_0,G}:\mathcal{F}(p_*(\omega^{k_0}_{\mathfrak{X}/T}),G)\rightarrow H^G_{N,h'}$$ such that $$\mathsf{M}^G_h(q)((\mathfrak{X}\rightarrow T),G,\rho)\simeq f_{\mathfrak{X}/T,k_0,G}^*((U_{N,h'}^G\rightarrow H^{G}_{N,h'}),G,\mathcal{B}_N),$$ where $N:=h(k_0)-1$, $h'(k):=h(k_0k)$.
\end{corollary}
\begin{remark}\label{CanM}
Given an isomorphism $ ((p:\mathfrak{X}_1\rightarrow T),G,\rho_1)\simeq  ((p:\mathfrak{X}_2\rightarrow T),G,\rho_2)$, we have an induced isomorphism $l:p_*(\omega^{k_0}_{\mathfrak{X}_1/T})\rightarrow p_*(\omega^{k_0}_{\mathfrak{X}_2/T})$ of $G$-sheaves on $T$ and both $p_*(\omega^{k_0}_{\mathfrak{X}_1/T})$ and $p_*(\omega^{k_0}_{\mathfrak{X}_2/T})$ have decomposition type $\beta$. Then $l$ induces a $C(G,\beta)$-equivariant isomorphism: $l_\mathcal{F}:\mathcal{F}(p_*(\omega^{k_0}_{\mathfrak{X}_1/T}),G)\rightarrow \mathcal{F}(p_*(\omega^{k_0}_{\mathfrak{X}_2/T}),G)$. From (\ref{Can}), (\ref{CanG}) and the proof of (\ref{prop4}) we have that $f_{\mathfrak{X}_1/T,k_0,G}= f_{\mathfrak{X}_2/T,k_0,G}\circ\l_\mathcal{F}$.  
\end{remark}
We have already shown that $\mathsf{M}^G_h$ is bounded (in the sense of \ref{cor1}). However in general $H^G_{N,h'}$ may not be a parametrizing space for $\mathsf{M}^G_h$, i.e., some fibre of $((U_{N,h'}^G\rightarrow H^{G}_{N,h'}),G,\mathcal{B}_N)$ may not be a canonical model. We will see that the set of points in $H^G_{N,h'}$ over which the fibre is a Gorenstein canonical model forms a locally closed subscheme $\bar{H}^G_{N,h'}$. In general such problems correspond to studying the local closedness of the moduli functor.
\begin{definition}[\cite{Kov09}, 5.C] \label{localclose}
A moduli functor of polarized varieties $\mathsf{M}$ is called {\em locally closed} if the following condition holds: For every flat family of polarized varieties $(\mathfrak{X}\rightarrow T,\mathfrak{L})$, there exists a locally closed subscheme $i:T'\hookrightarrow T$ such that if $f:S\rightarrow T$ is any morphism then $f^*(\mathfrak{X}\rightarrow T,\mathfrak{L})\in \mathsf{M}(S)$ iff $f$ factors through $T'$.
\end{definition}
Here we do not state a general "$G$-version" of local closedness, but only consider the case of Hilbert schemes. For a general discussion, see \cite{Kol08}, Corollary 24.
\begin{prop}\label{prop5}
Using the same notations as in (\ref{prop4}), there exists a locally closed subscheme $\bar{H}^G_{N,h'}$ of $H^G_{N,h'}$, satisfying the following conditions:\\
(1) $((\bar{U}_{N,h'}^G\rightarrow \bar{H}^{G}_{N,h'}),G,\mathcal{B}_N):=((U_{N,h'}^G\rightarrow H^{G}_{N,h'}),G,\mathcal{B}_N)|_{\bar{H}^G_{N,h'}}\in \mathsf{M}^G_h(\bar{H}^G_{N,h'})$.\\
(2) The morphism $f$ that we obtained in (\ref{prop4}) factors through $\bar{H}^G_{N,h'}$.
\end{prop}
\begin{proof}
In the case where $G$ is trivial the existence of $\bar{H}_{N,h'}$ follows from the facts that the subset  $\{x\in H_{N,h'}|\ (\omega^{k_0}_{U_{N,h'}/H_{N,h'}})_x\simeq (\mathcal{O}_{H_{N,h'}\times\mathbb{P}^N}(1)|_{U_{N,h'}})_x\}$ is closed in $H_{N,h'}$ (cf.\cite{Mum70}, II.5, Corollary 6) and being canonical and Gorenstein is an open property (cf.\cite{Elk81}).\\
In general we set $\bar{H}^{G,\beta}_{N,h'}:=\bar{H}_{N,h'}\bigcap H^{G,\beta}_{N,h'}$ 
and $\bar{H}^G_{N,h'}:= \bigsqcup \bar{H}^{G,\beta}_{N,h'}$. For condition (1), the fact that $(\bar{U}_{N,h'}\rightarrow \bar{H}_{N,h'})\in \mathsf{M}_h(\bar{H}_{N,h'})$ implies that $(\bar{U}_{N,h'}^G\rightarrow \bar{H}^{G}_{N,h'})\in \mathsf{M}_h(\bar{H}^{G}_{N,h'})$, now taking the action of $G$ into account, we have that $((\bar{U}_{N,h'}^G\rightarrow \bar{H}^{G}_{N,h'}),G,\mathcal{B}_N)\in \mathsf{M}^G_h(\bar{H}^G_{N,h'})$. Condition (2) is satisfied for similar reasons. 
\end{proof}
\begin{remark}\label{rk4}
(1) Given $(X_1,G,\rho_1),(X_2,G,\rho_2)\in \mathsf{M}^G_h(\mathbf{Spec}\mathbb{C})$ such that $H^0(\omega^{k_0}_{X_1})$ and $H^0(\omega^{k_0}_{X_2}) $ have the same decomposition type $\beta$, by (\ref{prop5}) there exist $f_i:\mathbf{Spec}(\mathbb{C})\rightarrow \bar{H}^{G,\beta}_{N,h'}$ such that $(X_i,G,\rho_i)\simeq\mathsf{M}^G_h(f_i)((\bar{U}_{N,h'}^{G,\beta}\rightarrow \bar{H}^{G,\beta}_{N,h'}),G,\beta)$ for $i=1,2$.\\
From the proof of (\ref{prop4}) we saw that $X_1$ and $X_2$ are isomorphic as $G$-marked varieties $\iff$ $\exists g\in C(G,\beta)$ such that $f_1(\mathbf{Spec}(\mathbb{C}))=\Psi_gf_2(\mathbf{Spec}(\mathbb{C}))$. \\
(2) Notations as in (\ref{cor1}). Assume that $T$ is connected and $p_*(\omega^{k_0}_{\mathfrak{X}/T})$ has decomposition type $\beta$ and denote by $\Psi'$ the action of $C(G,\beta)$ on $\mathcal{F}(p_*(\omega^{k_0}_{\mathfrak{X}/T}),G)$. From the proof of (\ref{prop2}) we see that $\forall g\in C(G,\beta)$, $\Psi'_g \times \Phi_{g}$ leaves $q^*\mathfrak{X}\simeq f_{\mathfrak{X}/T,k_0,G}^*(\bar{U}^{G,\beta}_{N,h'})$ invariant as a subscheme of $\mathcal{F}(p_*(\omega^{k_0}_{\mathfrak{X}/T}),G)\times \mathbb{P}^N$, i.e., $(\Psi'_g \times \Phi_{g})f_{\mathfrak{X}/T,k_0,G}^*(\bar{U}^{G,\beta}_{N,h'})=f_{\mathfrak{X}/T,k_0,G}^*(\bar{U}^{G,\beta}_{N,h'})$. This implies that 
$$(\Psi'_g \times id)f_{\mathfrak{X}/T,k_0,G}^*(\bar{U}^{G,\beta}_{N,h'})=f_{\mathfrak{X}/T,k_0,G}^*((id\times \Phi_{g^{-1}})(\bar{U}^{G,\beta}_{N,h'}))=f_{\mathfrak{X}/T,k_0,G}^*((\Psi_g\times id)(\bar{U}^{G,\beta}_{N,h'})).$$ Therefore we conclude that the morphism obtained in (\ref{cor1}), $$f_{\mathfrak{X}/T,k_0,G}:\mathcal{F}(p_*(\omega^{k_0}_{\mathfrak{X}/T}),G)\rightarrow \bar{H}^{G,\beta}_{N,h'},$$ is $C(G,\beta)$-equivariant.
\end{remark} 

\section{The Construction of $\mathfrak{M}_h[G]$}
In section 3 we have obtained a parametrizing space $\bar{H}^G_{N,h'}$ for the moduli functor $\mathsf{M}^G_h$, now we construct $\mathfrak{M}_h[G]$ as a quotient space of $\bar{H}^G_{N,h'}$ and show that it is the coarse moduli scheme for $\mathsf{M}^G_h$.\\
In (\ref{rk3}) we have seen that the group $C(G,\beta)$ acts on $H^{G,\beta}_{N,h'}$: it is clear that the subscheme $\bar{H}^{G,\beta}_{N,h'}$ is invariant under this action. The first goal of this section is to show that the quotient $\bar{H}^{G,\beta}_{N,h'}/C(G,\beta)$ exists (as a scheme).\\
Setting $SC(G,\beta):=SL(N+1,\mathbb{C})\bigcap C(G,\beta)$, it is easy to see that $\bar{H}^{G,\beta}_{N,h'}/C(G,\beta)\simeq \bar{H}^{G,\beta}_{N,h'}/SC(G,\beta)$ (if one of them exists).
Therefore from now on we consider $\bar{H}^{G,\beta}_{N,h'}/SC(G,\beta)$ instead. (It is not difficult to show that $SC(G,\beta)$ is reductive.)
\begin{lemma}\label{lem1}
$SC(G,\beta)$ acts properly on  $\bar{H}^{G,\beta}_{N,h'}$ and $\forall x\in \bar{H}^{G,\beta}_{N,h'}$, the stabilizer subgroup $Stab(x)$ is finite.
\end{lemma}
\begin{proof}
In the case where $G$ is trivial the lemma is known by studying the separatedness of the corresponding functor (cf. \cite{Vie95}, 7.6, 8.21; \cite{Kov09}, 5.D). Now since $SC(G,\beta)$ is a closed subgroup of $SL(N+1,\mathbb{C})$ and $\bar{H}^{G,\beta}_{N,h'}$ is a closed subscheme of $\bar{H}_{N,h'}$ which stays invariant under the action of $SC(G,\beta)$, the lemma follows immediately.
\end{proof}
In order to apply Geometric Invariant theory, we have to find an $SC(G,\beta)$-linearized invertible sheaf on $\bar{H}^{G,\beta}_{N,h'}$ and verify certain stability conditions (cf. \cite{MF82}, Chap.1). \\
Let us first look at the case where $G$ is trivial: let $p:\bar{U}_{N,h'}\rightarrow \bar{H}_{N,h'}$ be the universal family and 
define $$\lambda_{k_0}:=\det(p_*(\omega^{k_0}_{\bar{U}_{N,h'}/\bar{H}_{N,h'}})).$$ A result of Viehweg (see \cite{Vie95}, 7.17) states that $\lambda_{k_0}$ admits an $SL(N+1,\mathbb{C})$-linearization and $$\bar{H}_{N,h'}=(\bar{H}_{N,h'})^s(\lambda_{k_0}),$$ where $(\bar{H}_{N,h'})^s(\lambda_{k_0})$ denotes the set of $SL(N+1,\mathbb{C})$-stable points with respect to $\lambda_{k_0}$. Then it is easy to obtain the following proposition:
\begin{prop}
There exists a geometric quotient $(\mathfrak{M}^{G,\beta}_{k_0,h},\pi_{\beta})$ of $\bar{H}^{G,\beta}_{N,h'}$ by $SC(G,\beta)$, moreover:\\
(1) The quotient map $\pi_{\beta}: \bar{H}^{G,\beta}_{N,h'}\rightarrow \mathfrak{M}^{G,\beta}_{k_0,h}$ is an affine morphism.\\
(2) There exists an ample invertible sheaf $\mathcal{L}$ on $\mathfrak{M}^{G,\beta}_{k_0,h}$ such that $\pi_{\beta}^*\mathcal{L}\simeq (\lambda^{G,\beta}_{k_0})^n$ for some $n>0$, where setting $p_\beta:=p|_{\bar{U}_{N,h'}^{G,\beta}}:\bar{U}_{N,h'}^{G,\beta}\rightarrow\bar{H}^{G,\beta}_{N,h'},$  $\lambda^{G,\beta}_{k_0}:=\det((p_\beta)_*(\omega^{k_0}_{\bar{U}_{N,h'}^{G,\beta}/\bar{H}^{G,\beta}_{N,h'}}))$.
\end{prop}
\begin{proof}
Noting that $\omega^{k_0}_{\bar{U}_{N,h'}/\bar{H}_{N,h'}}|_{\bar{U}^{G,\beta}_{N,h'}}\simeq \omega^{k_0}_{\bar{U}_{N,h'}^{G,\beta}/\bar{H}^{G,\beta}_{N,h'}}$ (cf. \cite{HK04}, Lemma 2.6) and applying "cohomology and base change", we have that $\lambda^{G,\beta}_{k_0}\simeq\lambda_{k_0}|_{\bar{H}^{G,\beta}_{N,h'}}.$ Since $\bar{H}^{G,\beta}_{N,h'}$ (as a subscheme of $\bar{H}_{N,h'}$) is invariant under the $SC(G,\beta)$-action, the $SL(N+1,\mathbb{C})$-linearization of $\lambda_{k_0}$ induces a natural $SC(G,\beta)$-linearization of $\lambda^{G,\beta}_{k_0}$. By Lemma (\ref{lem1}), we have that $SL(N+1,\mathbb{C})$ acts properly on $\bar{H}_{N,h'}$ and $SC(G,\beta)$ acts properly on $\bar{H}^{G,\beta}_{N,h'}$. Noting that a one-parameter subgroup $\mu:\mathbb{C}^*\rightarrow SC(G,\beta)$ is also a subgroup of $SL(N+1,\mathbb{C})$ and that $\bar{H}^{G,\beta}_{N,h'}$ is closed in $\bar{H}_{N,h'}$, we see that for any $x\in \bar{H}^{G,\beta}_{N,h'}$, $\lim_{t\rightarrow 0}(\mu(t)x)$ exists in $\bar{H}^{G,\beta}_{N,h'}$ if and only if it exists in $\bar{H}_{N,h'}$. Now by applying the Hilbert-Mumford criterion (cf. \cite{MF82}, Theorem 2.1), we see that  $$(\bar{H}_{N,h'})^s(\lambda_{k_0})=\bar{H}_{N,h'}\Rightarrow (\bar{H}^{G,\beta}_{N,h'})^s(\lambda^{G,\beta}_{k_0})=\bar{H}^{G,\beta}_{N,h'}.$$ Then the proposition follows from standard GIT methods (cf. \cite{MF82}, Theorem 1.10).
\end{proof}
We are ready to prove the main theorem (\ref{thm}):
\begin{proof}[Proof of (\ref{thm})]
We set 
\begin{equation}\label{eq1}
\mathfrak{M}_h[G]:=\bigsqcup_{[\beta]\in\mathcal{B}_N}\mathfrak{M}^{G,\beta}_{k_0,h}
\end{equation}
 (note that if $\mathsf{M}^G_h(\mathbf{Spec}\mathbb{C})=\emptyset$ then $\mathfrak{M}_h[G]=\emptyset$).\\
Let us make the following convention: for any natural transformation $\theta:\mathsf{M}^G_h \rightarrow \Hom(-,Q)$, scheme $T$ and $[((p:\mathfrak{X}\rightarrow T),G,\rho)]\in \mathsf{M}^G_h(T)$, we write $\theta_T(\mathfrak{X})$ or simply $\theta(\mathfrak{X})$ as an abbreviation for $\theta_T([((p:\mathfrak{X}\rightarrow T),G,\rho)])$.\\
Step 1. Construction of a natural transformation $\eta:\mathsf{M}^G_h\rightarrow \Hom(-,\mathfrak{M}_h[G])$:\\
Given $T$ a scheme and $((p:\mathfrak{X}\rightarrow T),G,\rho)\in \mathsf{M}^G_h(T)$, it suffices to define $\eta$ on each connected component of $T$, hence we assume furthermore that $T$ is connected. We have the bundle of $G$-frames of $p_*(\omega^{k_0}_{\mathfrak{X}/T})$ over $T$, $q:\mathcal{F}(p_*(\omega^{k_0}_{\mathfrak{X}/T}),G)\rightarrow T$. By (\ref{cor1}) and (\ref{prop5}) there exists a  morphism $f_{\mathfrak{X}/T,k_0,G}:\mathcal{F}(p_*(\omega^{k_0}_{\mathfrak{X}/T}),G)\rightarrow \bar{H}^{G,\beta}_{N,h'}$ such that $$\mathsf{M}^G_h(q)((\mathfrak{X}\rightarrow T),G,\rho)\simeq \mathsf{M}^G_h(f_{\mathfrak{X}/T,k_0,G})((\bar{U}_{N,h'}^{G,\beta}\rightarrow \bar{H}^{G,\beta}_{N,h'}),G,\beta)$$ for some $[\beta] \in \mathcal{B}_N$. Setting $$\bar{f}_{\mathfrak{X}/T,k_0,G}:=\pi_{\beta}\circ f_{\mathfrak{X}/T,k_0,G}:\mathcal{F}(p_*(\omega^{k_0}_{\mathfrak{X}/T}),G)\rightarrow \mathfrak{M}^{G,\beta}_{k_0,h},$$ 
 by (\ref{rk4}-2) we see that $\bar{f}_{\mathfrak{X}/T,k_0,G}$ is $C(G,\beta)$-equivariant (where we take the trivial action on $\mathfrak{M}^{G,\beta}_{k_0,h}$). Since $T$ is the quotient of $\mathcal{F}(p_*(\omega^{k_0}_{\mathfrak{X}/T}),G)$ by $C(G,\beta)$ (cf. \ref{prop3}), there exists a (unique) morphism $\eta_T(\mathfrak{X}): T\rightarrow \mathfrak{M}_h[G]$ such that $\bar{f}_{\mathfrak{X}/T,k_0,G}= \eta_T(\mathfrak{X})\circ q$. Note that by (\ref{CanM}) $\eta_T(\mathfrak{X})$ is independent of the representative family $((p:\mathfrak{X}\rightarrow T),G,\rho)$ that we choose, hence $\eta_T(\mathfrak{X})$ is well defined. \\
 In order to show that $\eta$ is a natural transformation, let $l\in \Hom(S,T)$ and $((p:\mathfrak{X}\rightarrow T),G,\rho)\in \mathsf{M}^G_h(T)$: it suffices to show that $\eta_S(\mathfrak{X}_S)=\eta_T(\mathfrak{X})\circ l$. Without loss of generality we assume that $S$ and $T$ are connected and $p_*(\omega^{k_0}_{\mathfrak{X}/T})$ has decomposition type $\beta$, now considering the following commutative diagram: 
$$\begin{CD}
\mathcal{F}((p_{S})_*(\omega^{k_0}_{\mathfrak{X}_S/S}),G) @>\tilde{l}>>   \mathcal{F}(p_*(\omega^{k_0}_{\mathfrak{X}/T}),G)\\ 
@VVq_SV                           @VVqV\\
S               @>l>>             T
\end{CD},$$
from (\ref{Can}-1) and (\ref{CanG}) we see that $\bar{f}_{\mathfrak{X}_S/S,k_0,G}=\bar{f}_{\mathfrak{X}/T,k_0,G}\circ \tilde{l}$. Since $\bar{f}_{\mathfrak{X}_S/S,k_0,G}$,
$\bar{f}_{\mathfrak{X}/T,k_0,G}$ and $ \tilde{l}$ are all $C(G,\beta)$-equivariant, hence we have $\eta_S(\mathfrak{X}_S)=\eta_T(\mathfrak{X})\circ l$ by (\ref{prop3}). 
\\
Step 2. $\mathfrak{M}_h[G]$ is the coarse moduli scheme for $\mathsf{M}^G_h$:\\
(1) $\eta_{\mathbf{Spec}\mathbb{C}}$ induces a one-to-one correspondence between $\mathsf{M}^G_h(\mathbf{Spec}\mathbb{C})$ and the set of (closed) points of $\mathfrak{M}_h[G]$.\\
Surjectivity follows from (\ref{prop5}), and injectivity follows from (\ref{rk4}-1). \\
(2) The universal property of $\eta$.\\
Let $\theta:\mathsf{M}^G_h\rightarrow \Hom(-,Q)$ be another natural transformation: we show that there exists a unique morphism $\gamma:\mathfrak{M}_h[G]\rightarrow Q$ such that $\theta=\Hom(\gamma)\circ \eta$.\\
For any $[\beta] \in \mathcal{B}_N$, the universal family $((\bar{U}_{N,h'}^{G,\beta}\rightarrow \bar{H}^{G,\beta}_{N,h'}),G,\beta)\in \mathsf{M}^G_h(\bar{H}^{G,\beta}_{N,h'})$ induces a morphism $\theta_{\bar{H}^{G,\beta}_{N,h'}}(\bar{U}_{N,h'}^{G,\beta}): \bar{H}^{G,\beta}_{N,h'}\rightarrow Q$. For any $g\in C(G,\beta)$, we have that 
$$(\Psi_g\times id_{\mathbb{P}^N})((\bar{U}_{N,h'}^{G,\beta}\rightarrow \bar{H}^{G,\beta}_{N,h'}),G,\beta)= (id_{\bar{H}^{G,\beta}_{N,h'}}\times \Phi_{g^{-1}})((\bar{U}_{N,h'}^{G,\beta}\rightarrow \bar{H}^{G,\beta}_{N,h'}),G,\beta)$$ 
as subschemes of $\bar{H}^{G,\beta}_{N,h'}\times \mathbb{P}^N$, noting that the right hand side is isomorphic to $((\bar{U}_{N,h'}^{G,\beta}\rightarrow \bar{H}^{G,\beta}_{N,h'}),G,\beta)$ as $G$-marked families, we see that $$\theta_{\bar{H}^{G,\beta}_{N,h'}}(\bar{U}_{N,h'}^{G,\beta})=\theta_{\bar{H}^{G,\beta}_{N,h'}}(\bar{U}_{N,h'}^{G,\beta})\circ \Psi_g.$$ This implies that $\theta_{\bar{H}^{G,\beta}_{N,h'}}(\bar{U}_{N,h'}^{G,\beta})$ is $C(G,\beta)$-equivariant, hence it induces a (unique) morphism $\gamma_{\beta}:\mathfrak{M}^{G,\beta}_{k_0,h}\rightarrow Q$ such that $\theta_{\bar{H}^{G,\beta}_{N,h'}}(\bar{U}_{N,h'}^{G,\beta})=\gamma_{\beta}\circ \eta_{\bar{H}^{G,\beta}_{N,h'}}(\bar{U}_{N,h'}^{G,\beta})$. Now we can define $\gamma:\mathfrak{M}_h[G]\rightarrow Q$ such that the restriction of $\gamma$ to each $\mathfrak{M}^{G,\beta}_{k_0,h}$ is $\gamma_{\beta}$. \\
From the construction of $\gamma$ we already saw that $\gamma$ must be unique, it remains to show that $\theta=\Hom(\gamma)\circ \eta$. Given $((p:\mathfrak{X}\rightarrow T),G,\rho)\in \mathsf{M}^G_h(T)$, let $q:\mathcal{F}(p_*(\omega^{k_0}_{\mathfrak{X}/T}),G)\rightarrow T$ be the bundle of $G$-frames of $p_*(\omega^{k_0}_{\mathfrak{X}/T})$, we assume again that $T$ is connected and $p_*(\omega^{k_0}_{\mathfrak{X}/T})$ has decomposition type $\beta$. By (\ref{cor1}) and  (\ref{prop5}) there exists $f_{\mathfrak{X}/T,k_0,G}:\mathcal{F}(p_*(\omega^{k_0}_{\mathfrak{X}/T}),G)\rightarrow \bar{H}^{G,\beta}_{N,h'}$ such that $$\mathsf{M}^G_h(q)((\mathfrak{X}\rightarrow T),G,\rho)\simeq\mathsf{M}^G_h(f_{\mathfrak{X}/T,k_0,G})((\bar{U}_{N,h'}^{G,\beta}\rightarrow \bar{H}^{G,\beta}_{N,h'}),G,\beta),$$ hence we have that $$\theta(q^*\mathfrak{X})=\theta_{\bar{H}^{G,\beta}_{N,h'}}(\bar{U}_{N,h'}^{G,\beta})\circ f_{\mathfrak{X}/T,k_0,G}=\gamma_{\beta}\circ \eta_{\bar{H}^{G,\beta}_{N,h'}}(\bar{U}_{N,h'}^{G,\beta})\circ f_{\mathfrak{X}/T,k_0,G}= \gamma_{\beta}\circ \eta(q^*\mathfrak{X}),$$ where the first and third equalities hold since $\theta$ and $\eta$ are natural transformations, the second equality holds by the construction of $\gamma_{\beta}$. Finally the fact that $f_{\mathfrak{X}/T,k_0,G}$ and $\theta_{\bar{H}^{G,\beta}_{N,h'}}(\bar{U}_{N,h'}^{G,\beta})$ are $C(G,\beta)$-equivariant $\Rightarrow$ $\theta(q^*\mathfrak{X})$ is also $C(G,\beta)$-equivariant. By (\ref{prop3}) $\exists !$ $l'\in \Hom(T,Q)$ such that $\theta(q^*\mathfrak{X})=l'\circ q$, which implies that $\theta_T(\mathfrak{X})=l'=\gamma_{\beta}\circ \eta_T(\mathfrak{X})$.
\end{proof}
As an application of our results, we show that the locus $\mathfrak{M}_h(G)$ inside $\mathfrak{M}_h$ of varieties which admit an effective action by a group $G$ is closed. This has been proven in \cite{Cat83}, Theorem 1.8 for the case of surfaces, the idea there generalizes naturally to higher dimensional cases.\\
Given a faithful representation $\beta:G\rightarrow GL(N+1,\mathbb{C})$, we have a natural inclusion $i_\beta:\bar{H}^{G,\beta}_{N,h'}\subset \bar{H}_{N,h'}$. Noting that the restriction of the quotient map $\pi:\bar{H}_{N,h'}\rightarrow \mathfrak{M}_h$ to $\bar{H}^{G,\beta}_{N,h'}$ is $SC(G,\beta)$-equivariant, we obtain an induced morphism $u^{G,\beta}_{k_0,h}:\mathfrak{M}^{G,\beta}_{k_0,h}\rightarrow \mathfrak{M}_h$. We define a morphism $u^G_h:\mathfrak{M}_h[G]\rightarrow \mathfrak{M}_h$ such that $u^G_h|_{\mathfrak{M}^{G,\beta}_{k_0,h}}=u^{G,\beta}_{k_0,h}$. We denote by $\mathfrak{M}_h(G)$ the (scheme-theoretic) image of $u^G_h$ in $\mathfrak{M}_h$. Then we can interpret the problem into showing that $u^G_h$ maps $\mathfrak{M}_h[G]$ surjectively onto $\mathfrak{M}_h(G)$.
\begin{corollary}
The morphism $u^G_h:\mathfrak{M}_h[G]\rightarrow \mathfrak{M}_h$ is finite and maps $\mathfrak{M}_h[G]$ surjectively onto $\mathfrak{M}_h(G)$; $\mathfrak{M}_h(G)$ is a closed subscheme of $\mathfrak{M}_h$.
\end{corollary}
\begin{proof}
It is easy to see that $u^G_h$ is quasi-finite: given a point $[X]\in \mathfrak{M}_h$, since $\Aut(X)$ is finite, then the set of injective homomorphisms $\rho:G\rightarrow \Aut(X)$ is finite, hence $(u^G_h)^{-1}([X])$, which corresponds to the set of isomorphism classes of $G$-markings on $X$, is also finite.\\
For the remaining statements, it suffices to show that $u^G_h$ is proper, which is equivalent to showing that $u^{G,\beta}_{k_0,h}:\mathfrak{M}^{G,\beta}_{k_0,h}\rightarrow \mathfrak{M}_h$ is proper for each $[\beta]\in \mathcal{B}_N$. Applying the valuative criterion of properness, we have to prove that for every pointed curve $(C,O)$ (not necessarily complete) and for any commutative diagram 
$$\begin{CD}
C^{\star} @>f'>>   \mathfrak{M}^{G,\beta}_{k_0,h}\\ 
@VViV                           @VVu^{G,\beta}_{k_0,h}V\\
C               @>f>>            \mathfrak{M}_h
\end{CD}$$
where $C^\star:=C-\{O\}$, there exists a unique $l:C\rightarrow \mathfrak{M}^{G,\beta}_{k_0,h}$ making the whole diagram commute.\\
By GIT we know that $\mathfrak{M}^{G,\beta}_{k_0,h}$ is quasi-projective and hence separated, therefore the uniqueness of $l$ is clear. For the existence of $l$, since $\pi_{\beta}: \bar{H}^{G,\beta}_{N,h'}\rightarrow \mathfrak{M}^{G,\beta}_{k_0,h}$ is a quotient map of quasi-projective schemes, it suffices to show that there exists a finite morphism $v:(B,O')\rightarrow (C,O)$ and a morphism $l':B\rightarrow \bar{H}^{G,\beta}_{N,h'}$ such that 
$$(*)\ u^{G,\beta}_{k_0,h}\circ \pi_\beta\circ l'=f\circ v\ and\ \pi_\beta\circ (l'|_{B^\star})=f'\circ (v|_{B^\star}),$$ where $B^\star:=B-\{O'\}$.\\
Considering the quotient map $\pi:\bar{H}_{N,h'}\rightarrow \mathfrak{M}_h$, we can assume without loss of generality that we have a morphism $m:C\rightarrow \bar{H}_{N,h'}$ such that $f=\pi \circ m$. Then we obtain a family $(\mathfrak{X}\rightarrow C):=m^*(\bar{U}_{N,h'})\in \mathsf{M}_h(C)$ such that $\mathfrak{X}\subset C\times \mathbb{P}^N$. The idea of constructing the morphism $v:(B,O')\rightarrow (C,O)$ is similar to that of (\ref{prop2}). We consider first the subspace 
$$Z:=\{(t,A(t)|A(t)\mathfrak{X}_t\ corresponds\ to\ a\ point\ in\ \bar{H}^{G,\beta}_{N,h'}\}\subset C\times GL(N+1,\mathbb{C}).$$ By assumption we see that $p_1:Z-p^{-1}_1(O)\rightarrow C^\star$ is surjective, where $p_1:C\times GL(n+1,\mathbb{C})\rightarrow C$ is the projection onto the first factor, hence we can find a curve $B'$ inside $Z$, such that $p_1|_{B'}:B'\rightarrow C^\star$ is surjective. For similar reasons as in (\ref{prop2}), we get a $G$-marked family $((p_1|_{B'})^*\mathfrak{X}^\star\rightarrow B'),G,\beta)$, where $\mathfrak{X}^\star:=\mathfrak{X}-\mathfrak{X}_O$. After possibly taking the normalization of $B'$, we can extend the morphism $p_1|_{B'}$ to a morphism $v:(B,O')\rightarrow (C,O)$ and we see that $(((v|_{B^\star})^*\mathfrak{X}^\star\rightarrow B^\star),G,\beta)$ is a $G$-marked family, where $B^\star:=B-\{O'\}$. \\
We claim that the $G$-action on $(v|_{B^\star})^*\mathfrak{X}^\star\rightarrow B^\star$ can be extended to an action on $(\mathfrak{X}'\rightarrow B):=(v^*\mathfrak{X}\rightarrow B)$. Since $\omega^{k_0}_{\mathfrak{X}'/B}$ induces an embedding $i:\mathfrak{X}'\rightarrow B\times \mathbb{P}^N$, we see that the claim is equivalent to that $i(\mathfrak{X}')$ is invariant under the action $\pi_2^*(\beta)$, where $\pi_2: B\times \mathbb{P}^N\rightarrow \mathbb{P}^N$ is the projection onto the second factor. After possibly shrinking $B$, we can assume that $B$ is connected and hence $i(\mathfrak{X}') $ is irreducible. The fact that $((v|_{B^\star})^*\mathfrak{X}^\star\rightarrow B^\star),G,\beta)$ is a $G$-marked family implies that $i((v|_{B^\star})^*\mathfrak{X}^\star)$ is invariant under $\pi_2^*(\beta)$, now from the irreducibility of $i(\mathfrak{X}')$ we see that $i(\mathfrak{X}')$ is also invariant under the action $\pi_2^*(\beta)$.\\
Now we have a $G$-marked family $((\mathfrak{X}'\rightarrow B),G,\beta)$, by (\ref{prop4}) we obtain a morphism $l':B \rightarrow \bar{H}^{G,\beta}_{N,h'}$, it is easy to check that $l'$ satisfies $(*)$.
\end{proof}
\section{Decompositions of $\mathfrak{M}_h[G]$ }
In the proof of theorem \ref{thm} (cf. \ref{eq1}) we saw that $\mathfrak{M}_h[G]$ has a decomposition which depends upon the choice of a sufficiently large natural number $k$: $$\mathcal{D}^G_{k,h}:~\mathfrak{M}_h[G]=\bigsqcup_{[\beta]\in\mathcal{B}_N}\mathfrak{M}^{G,\beta}_{k,h}.$$
Since given a $G$-marked family $((\mathfrak{X}\rightarrow T),G,\rho)$ over a connected base $T$, the induced $G$-representations on $H^0(\omega^k_{\mathfrak{X}_t})$ are all isomorphic to each other for any $t\in T$ (cf. \cite{Cat13}, Prop 37), each component of the decomposition is a union of connected components of $\mathfrak{M}_h[G]$.
\begin{definition}\label{CanDecomp}
$(1)$ Given a space $X$ with two decompositions $~\mathcal{D}_1:X=\bigsqcup_{i\in I}Y_i$ and $\mathcal{D}_2:X=\bigsqcup_{j\in J}W_j$, where each $Y_i$, $W_j$ is a union of connected components of $X$, their minimal refinement is defined as: $$\mathcal{D}_1\cap \mathcal{D}_2: X=\bigsqcup_{(i,j)\in I\times J}Y_i\cap W_j.$$
$(2)$ The {\em canonical representation type decomposition} of $\mathfrak{M}_h[G]$ is the minimal refinement of all the above decompositions:
$$\mathcal{D}_h[G]:=\cap_{k\in \mathcal{K}}\mathcal{D}^G_{k,h},$$
where $\mathcal{K}$ denotes the set of natural numbers satisfying Matsusaka's big theorem (cf. \cite{Mat86}, Theorem 2.4).
\end{definition}
\begin{remark}
Since $\mathfrak{M}_h[G]$ is a quasi-projective scheme, we see immediately that there exists 
a minimal natural number $N(h,G)$ and integers $k_1,...,k_{N(h,G)}$ such that
 $$\mathcal{D}_h[G]=\cap^{N(h,G)}_{i=1}\mathcal{D}^G_{k_i,h},$$
\end{remark}

\noindent
Several natural questions arise:\\
\underline{Question $1$.} What is an explicit bound for $N(h,G)$? \\
\underline{Question $2$.} Are the components of $\mathcal{D}_h[G]$ connected? or how many connected components do they have?\\

To answer question $1$, we provide first a method which works in general,  the main idea is to consider suitable Hilbert resolutions of the canonical rings of varieties with a fixed Hilbert polynomial $h$ (cf. \cite{Cat92}, Section 2). Then we look at the case of algebraic curves and obtain a more precise bound.\\
Since the functor $\mathsf{M}_h$ is bounded, there exists a minimal natural number $m=m(h)$ such that $\forall X\in \mathsf{M}_h(\mathbf{Spec}(\mathbb{C}))$, $H^i(X,\omega^m_X)=0$ for any $i>0$ and the $m$-th pluricanonical map of $X$, $\phi_m:X\rightarrow \mathbb{P}^n$ , is an embedding, where $n:=h(m)-1$. Recall that the canonical ring of $X$ is:
$$\mathcal{R}=\mathcal{R}(X,\omega_X):=\bigoplus_{k\geq 0}H^0(X,\omega^k_X)$$
Since $\omega_X$ is ample, $\mathcal{R}$ is a finite graded module over the graded ring $\mathcal{A}:=Sym(H^0(X,\omega^m_X))$. The degree $k$ direct summand of $\mathcal{R}$ (resp. $\mathcal{A}$) is denoted by $\mathcal{R}_k$ (resp. $\mathcal{A}_k$).
\begin{remark}\label{actionRA}
Assuming a group $G$ acts on $X$, we have naturally induced actions on $\mathcal{R}$ and $\mathcal{A}$. It is easy to see that these actions are compatible in the following sense:\\
(1) $\forall k\in \mathbb{N}$, $\mathcal{A}_k$ (resp. $\mathcal{R}_k$) is a $G$-invariant subspace of $\mathcal{A}$ (resp. $\mathcal{R}$).\\
(2) $\forall g\in G,a_{k_1}\in \mathcal{A}_{k_1}$ and $a_{k_2}\in \mathcal{A}_{k_2}$, $g(a_{k_1}a_{k_2})=(ga_{k_1})(ga_{k_2})$ (the same holds for $\mathcal{R}$). \\
(3) $\forall g\in G, a_{k_1}\in \mathcal{A}_{k_1}$ and $u_{k_2}\in \mathcal{R}_{k_2}$, $g(a_{k_1}u_{k_2})=(ga_{k_1})(gu_{k_2})$.
\end{remark}
Denoting by $\delta$ the depth of $\mathcal{R}$ as an $\mathcal{A}$-module, by Hilbert's syzygy theorem we have a minimal free resolution of $\mathcal{R}$ of length $n+1-\delta$ (cf. \cite{Gre89}, Theorem 1.2):
$$0\rightarrow L_{n+1-\delta}\rightarrow L_{n-\delta}\rightarrow \dots \rightarrow L_1\rightarrow L_0\rightarrow \mathcal{R}\rightarrow 0.$$ 
Now taking the action of $G$ into account, we have the following:
\begin{lemma}\label{G_min_res}
Let $\mathcal{A}=\mathbb{C}[x_0,...,x_n]$ and let $\mathcal{M}$ be a finite graded $\mathcal{A}$-module. Assuming that we have actions of $G$ on $\mathcal{A}$ and $\mathcal{M}$ such that \ref{actionRA} $(1)$, $(2)$ and $(3)$ are satisfied, then there exists a minimal $G$-equivariant free resolution of $\mathcal{M}$:
$$0\rightarrow L_{n+1-\delta}\rightarrow L_{n-\delta}\rightarrow \dots \rightarrow L_1\rightarrow L_0\rightarrow \mathcal{M}\rightarrow 0,$$
where $\delta$ is the depth of $\mathcal{M}$ as an $\mathcal{A}$-module. Moreover, $L_i$ is a direct sum: $$L_i=\bigoplus_{\chi\in Irrchar(G)}\bigoplus^{s_\chi}_{j=1}\mathcal{A}(-n_{\chi,i,j})\otimes V_\chi,$$
where $Irrchar(G)$ denotes the set of irreducible characters of $G$ and $V_\chi$ is the irreducible representation associated to $\chi$.
\end{lemma}
\begin{proof}
Since $\mathcal{M}$ is a finitely generated $\mathcal{A}$-module, there exists a minimal integer $k_1$ such that $\mathcal{M}_{k_1}\neq 0$. We have a natural $G$-equivariant $\mathcal{A}$-module morphism: 
$$\psi_1:\mathcal{A}(-k_1)\otimes \mathcal{M}_{k_1}\rightarrow \mathcal{M},$$ 
where the action on the left hand side is: $g(a\otimes m)=(ga)\otimes (gm)$. \\
Now $\mathcal{M}/Im(\psi_1)$ is again a finitely generated graded $\mathcal{A}$-module and $G$-module, hence we have $\bar{\eta_2}:\mathcal{A}(-{k_2})\otimes (\mathcal{M}/Im(\psi_1))_{k_2}\rightarrow \mathcal{M}/Im(\psi_1)$, which can be lifted to a $G$-equivalent homomorphism $\eta_2:\mathcal{A}(-{k_2})\otimes \mathcal{M}'_{k_2}\rightarrow \mathcal{M}$, where $k_2>k_1$ is the minimal integer such that $(\mathcal{M}/Im(\psi_1))_{k_2}\neq 0$, and $\mathcal{M}'_{k_2}$ is a $G$-invariant subspace of $\mathcal{M}_{k_2}$ which maps isomorphically onto $(\mathcal{M}/Im(\psi_1))_{k_2}$. We repeat the process and (since $\mathcal{M}$ is a finitely generated ) after a finite number of steps we obtain $L_0=\oplus^{l_0}_{\nu=1}\mathcal{A}(-k_\nu)\otimes \mathcal{M}'_{k_\nu}$ (we set $\mathcal{M}'_{k_1}=\mathcal{M}_{k_1}$) and a surjective $G$-equivariant morphism $d_0:L_0\twoheadrightarrow \mathcal{M}$. By decomposing $\mathcal{M}'_{k_j}$ into irreducible $G$-subspaces we get the promised form of $L_0$. From our construction, we see that $L_0$ is a finitely generated graded-$\mathcal{A}$-module and $G$-module satisfying \ref{actionRA} $(3)$.\\
We define $L_i$ and $d_i$ for $i\geq 1$ inductively: assuming that we already have $d_{i-1}:L_{i-1}\rightarrow L_{i-2}$ and $\ker(d_{i-1})$ is a finitely generated graded $\mathcal{A}$-module and $G$-module satisfying \ref{actionRA} $(3)$, we repeat the construction process of $d_0$ and get $d_i:L_i\twoheadrightarrow \ker(d_{i-1})$. By Hilbert's syzygy theorem we have $L_{n-\delta+2}=0$. From our construction it is clear that the resulting resolution is minimal.
\end{proof}
Setting $N'(h,G):=m+\max\{n_{\chi,i,j}\}$, from (\ref{G_min_res}) we have the following:
\begin{prop}
For any $k>N'(h,G)$, the $G$-representation on $\mathcal{R}_k$ is determined by the representations on $\mathcal{R}_1,...,\mathcal{R}_{N'(h,G)}$, hence $N(h,G)\leq N'(h,G)$.
\end{prop}

In order to find an explicit bound on $N(h,G)$, we estimate the integers $m$ and $\max\{n_{\chi,i,j}\}$ separately.\\

The problem of finding an effective bound on $m$ is the so called "effective Matsusaka problem". Koll{\'a}r has shown in \cite{Kol93} that $m\leq 2(d+3)(d+2)!(2+d)$, where $d:=deg(h)=dimX$. If we only consider canonically polarized manifolds, we have better results (cf. \cite{Dem96}, \cite{Siu02}): we would like to mention the result by Angehrn and Siu, they have shown that $m\leq (d+1)(d^2+d+4)/2+2$ (cf.\cite{AS95}).\\

To determine $\max\{n_{\chi,i,j}\}$, we recall first the notion of the Castelnuovo-Mumford regularity (cf. \cite{Mum66}, Lecture 14).
\begin{definition}
Let $\mathcal{F}$ be a coherent sheaf on $\mathbb{P}^n$: $\mathcal{F}$ is said to be  {\em $s$-regular} if $H^i(\mathbb{P}^n,\mathcal{F}(s-i))=0$ for all $i>0$, the {\em regularity} of $\mathcal{F}$ is the minimal natural number with this property.\\
The regularity of a graded $\mathcal{A}$-module $\mathcal{M}$ is the regularity of its associated sheaf $\widetilde{\mathcal{M}}$.
\end{definition}

Let $s$ be the regularity of $\mathcal{R}$ as an $\mathcal{A}$-module: we have the following inequalities.
\begin{lemma}
Notations as in (\ref{G_min_res}). For any $i,j$ and $\chi$, $i\leq n_{\chi,i,j}\leq i+s$.
\end{lemma}
\begin{proof}
See \cite{Cat92}, Section 2.
\end{proof}
An immediate consequence is that
\begin{prop}
$\max\{n_{\chi,i,j}\}\leq s+n+1-\delta\leq s+n+1$.
\end{prop} 
We refer to \cite{Mum66}, Lecture 14 for the fact that given a Hilbert polynomial $h$, $\forall [X]\in \mathsf{M}_h(\mathbf{Spec}(\mathbb{C}))$, the regularity of $\mathcal{R}(X,\omega_X)$ (as an $\mathcal{A}$-module) is bounded by a polynomial in the coefficients of $h(mx)$.\\

Observe that the ring $\mathcal{R}$ is a direct sum of graded $\mathcal{A}$-submodules: $\mathcal{R}=\bigoplus^{m-1}_{j=0}\mathcal{R}(j),$ where $\mathcal{R}(j):=\bigoplus_{i\geq 0}\mathcal{R}_{j+mi}$. Hence we have the following proposition:
\begin{prop}
For large $k$, the $G$-representation on $\mathcal{R}_k$ is determined by the representation on $\mathcal{R}_m$ and the representations on the lower degree summands $\mathcal{R}_l$ whose degree $l$ lies in the same modulo $m$ congruence class of $k$. 
\end{prop}
In the rest of this section we answer question $1$ and question $2$ for curves using topological methods.
In the case of curves we use genera instead of Hilbert polynomials, for instance, for curves of genus $g\geq 2$ the corresponding moduli space is denoted by $\mathfrak{M}_g[G]$.
\begin{definition}
Given a vector $v=(a_1,b_1,...,a_g,b_g;c_1,...c_r)\in G^{2g+r}$, we call $v$ a {\em $G$-Hurwitz vector} of type $(m_i)^r_{i=1}$ if the following conditions are satisfied:\\
$(1)$ $Ord(c_j)=m_i>1$;\\
$(2)$ $\prod^{g}_{i=1}[a_i,b_i]\prod^r_{j=1}c_j=1$;\\
$(3)$ the entries of $v$ generate the group $G$.
\end{definition}
From now on $C$ shall denote a smooth projective curve of genus $g\geq 2$. Moreover we assume that a finite group $G$ acts effectively on $C$, we denote by $C'$ the quotient curve $C/G$ and by $g'$ the genus of $C'$. The Galois cover $p:C\rightarrow C'$ is branched in $r$ points ($r=0$ if $p$ is unramified) on $C'$ with branching indices $m_1,...,m_r$. The cover $p$ has an associated Hurwitz vector $v=(a_1,b_1,...a_{g'},b_{g'};c_1,...c_r)$ of type $(m_i)^r_{i=1}$ (cf. \cite{Cat15}, 11.3).\\
The main ingredient in comparing $\mathcal{D}^G_{k,g}$ for different $k$'s is the following version of the Chevalley-Weil formula:
\begin{theorem}[Chevalley-Weil, cf.\cite{CW34}]\label{CW}
Let $(C,G,\phi)$ be a $G$-marked curve and let $v=(a_1,b_1,...a_{g'},b_{g'};c_1,...c_r)$ be a Hurwitz vector associated to the cover $C\rightarrow C/\phi(G)$. \\
Denote by $\chi_{\phi_k}$ the character of the representation $\phi_k:G\rightarrow H^0(\omega^k_C)$ which is induced naturally by $\phi$, and let $\chi_\rho$ be the character of an irreducible representation $\rho:G\rightarrow GL(W_\rho)$. We have the following formulae:
$$(1)~\langle\chi_{\phi_1},\chi_\rho\rangle=\chi_\rho(1_G)(g'-1)+\sum^r_{i=1}\sum^{m_i-1}_{\alpha=1}\frac{\alpha N_{i,\alpha}}{m_i}+\sigma,$$
where setting $\xi_{m_i}:=\exp(2\pi i/{m_i})$, $N_{i,\alpha}$ is the multiplicity of $\xi^{\alpha}_{m_i}$ as eigenvalue of $\rho(c_i)$, and $\sigma=1$ if $\rho$ is trivial, otherwise $\sigma=0$. 
$$(2)~\langle\chi_{\phi_k},\chi_\rho\rangle=\frac{2k}{|G|}\chi_\rho(1_G)(g-1)-\chi_\rho(1_G)(g'-1)-\sum^r_{i=1}\sum^{m_i-1}_{\alpha=0}N_{i,\alpha}\frac{[-\alpha-k]_{m_i}}{m_i}.$$
Here $k\geq 2$ , $[n]_{m_i}\in \{0,...,m_i-1\}$ is the congruence class of the integer $n$ modulo $m_i$.
\end{theorem}
\begin{remark}
The Chevalley-Weil formula given in (\ref{CW}) is not in the original form of \cite{CW34}, but in the form of \cite{FG15}, theorem 1.11.\footnote{The authors proved the formula for $k=1$, but with their method one easily obtains the formula for any $k$. I am thankful to C. Glei{\ss}ner for bringing these formulae to my attention.}
\end{remark}
Using (\ref{CW}), we see that $$\langle\chi_{\phi_{k+|G|}},\chi_\rho\rangle-\langle\chi_{\phi_k},\chi_\rho\rangle=2\chi_\rho(1_G)(g-1)\text{ for }k\geq2$$ and $$\langle\chi_{\phi_{1+|G|}},\chi_\rho\rangle-\langle\chi_{\phi_1},\chi_\rho\rangle=2\chi_\rho(1_G)(g-1)-\sigma ,$$ which is independent of the action $\phi$. Hence we have the following:
\begin{corollary}\label{period}
For curves of genus $g\geq 2$ and a finite group $G$, $\mathcal{D}^G_{k,g}$ and $\mathcal{D}^G_{k+|G|,g}$ give the same decomposition of $\mathfrak{M}_g[G]$ for any $k\geq 1$. Therefore we have $\mathcal{D}_g[G]=\cap^{|G|}_{k=1}\mathcal{D}^G_{k,g}$ and $N(g,G)\leq |G|$.
\end{corollary}
The next example shows that $\mathcal{D}^G_{k_1,g}$ and $\mathcal{D}^G_{k_2,g}$ could be different if $|k_1-k_2|< |G|$, hence $\mathcal{D}_g[G]$ might be a proper refinement of each $\mathcal{D}^G_{k,g}$.
\begin{example}\label{diffk}
In this example $G=\mathbb{Z}/3\mathbb{Z}=\{0,\bar{1},\bar{2}\}$ and $g=6$. Let $\chi_i:G\rightarrow \mathbb{C}^*, \bar{1}\mapsto \xi^i_3$, $i=1,2$ be the nontrivial irreducible characters of $G$. Consider two $G$-marked curves $(C_1,G,\phi)$ and $(C_2,G,\phi')$ of genus $g$ with associated Hurwitz vectors 
$$v=(\bar{1},0,0,\bar{2};\bar{2},\bar{1})\text{ and}$$
$$v'=(\bar{1},\bar{1},\bar{2},\bar{2},\bar{1},\bar{1},\bar{2},\bar{2}).$$
Using (\ref{CW}), one computes easily that $\chi_{\phi_1}=2\chi_{triv}+2\chi_1+2\chi_2$, $\chi_{\phi_2}=5\chi_{triv}+5\chi_1+5\chi_2$, $\chi_{\phi_3}=9\chi_{triv}+8\chi_1+8\chi_2$; $\chi_{\phi'_1}=3\chi_1+3\chi_2$, $\chi_{\phi'_2}=5\chi_{triv}+5\chi_1+5\chi_2$ and $\chi_{\phi'_3}=11\chi_{triv}+7\chi_1+7\chi_2$. Hence we see that $\mathcal{D}^G_{2,6}$ is different from $\mathcal{D}^G_{1,6}$ and $\mathcal{D}^G_{3,6}$.
\end{example}

We answer now Question $2$. The idea is to consider the topological types of $G$-actions on curves. The following observations are important: 
\begin{remark}\label{Observatons}
\begin{itemize}
\item[•]
\item[1)]Using (\ref{CW}), we see that if two $G$-marked curves have the same marked\footnote{The word "marked" means that we do not allow automorphisms of $G$.} topological type (i.e., the equivalence class of the topological $G$-actions on a compact Riemann surface with a given genus, cf. \cite{Cat15}, 11.2), then they must have the same representation type for all $k\geq 1$.
\item[2)]Given a $G$-marked family of curves over a connected base, then the marked topological types are all the same for any $G$-marked curve in the family (cf. \cite{Cat15}, chapter 11).
\item[3)]Given two $G$-marked curves of genus $g$ such that $G$ acts freely on both curves, from (\ref{CW}) we see that the respectively induced $G$-representations on $H^0(\omega^k)$ are the same for all $k$. Moreover both representations on $H^0(\omega^k)$ are direct sum of regular $G$-representations for $k\geq 2$.
\end{itemize}
\end{remark}
\begin{definition}\label{regular component}
(1) Given a finite group $G$, we denote by $\chi_{r.r}$ the character of the {\em regular representation} of $G$ on the group algebra $\mathbb{C}[G]$ induced by left translation of $G$.\\
(2) The {\em component of the regular representation} $(\mathfrak{M}_g[G])_{r.r}$ of $\mathfrak{M}_g[G]$ (with respect to the decomposition $\mathcal{D}_g[G]$) is the subscheme of $\mathfrak{M}_g[G]$ consisting of the $G$-marked curves $[(C,G,\phi)]$ of genus $g$ such that there exists a sequence of natural numbers $\{n_k\}$, such that $\chi_{\phi_k}=n_k\chi_{r.r}$ for all $k\geq 2$.
\end{definition}
Recall that a split metacyclic group $G$ is a split extension of two cyclic groups, or equivalently $G$ has the following presentation:
$$G=\langle x,y|x^m=y^n=1,yxy^{-1}=x^r\rangle$$
 where $m,n$ and $r$ are positive integers such that $r^n\equiv 1\mod m$. \\

Using the above observations and assuming $G$ is a nonabelian split metacyclic group, we give a lower bound for the number of connected components of $(\mathfrak{M}_g[G])_{r.r}$.\\

Denote by $\mathfrak{MTF}(G,g)$ the set of marked topological types of {\em free} $G$-actions on a compact Riemann surface of genus $g$. By \ref{Observatons}. 2) and 3) we see that $(\mathfrak{M}_g[G])_{r.r}$ has at least $|\mathfrak{MTF}(G,g)|$ connected components.\\

In the case that $G$ is a nonabelian split metacyclic group, we have the following result of Edmonds.
\begin{theorem}[\cite{Edm83}, Theorem 1.7]\label{Thm_edm}
Given $G$ a nonabelian split metacyclic group, then there is a bijection $\mathbf{B}: \mathfrak{MTF}(G,g)\rightarrow H_2(G,\mathbb{Z})$.
\end{theorem}
With our preceding discussion, we immediately have the following:
\begin{prop}
Let $G$ be a nonabelian split metacyclic group: $\forall g\geq 2$, $(\mathfrak{M}_g[G])_{r.r}$ has at least $|H_2(G,\mathbb{Z})|$ connected components.
\end{prop}
In the end we provide a formula to compute $H_2(G,\mathbb{Z})$ for a split metacycilc group $G$. 
 \begin{lemma}
 $H_2(G)=\mathbb{Z}/d\mathbb{Z}$, where $d=\frac{\gcd(m,~r-1)\gcd(m,~\Sigma^{n-1}_{i=0}r^i)}{m}$ .
 \end{lemma}
 \begin{proof}
 See \cite{Edm83}, Lemma $1.2$.
 \end{proof}
\subsection*{Acknowledgments}
The author is currently sponsored by the project "ERC Advanced Grand 340258 TADMICAMT", part of the work took place in the realm of the DFG Forschergruppe 790 "Classification of algebraic surfaces and compact complex manifolds".\\
The author would like to thank Fabrizio Catanese for suggesting the topic of this paper, for many inspiring discussions with the author and for his encouragement to the author. The author would also like to thank Christian Glei{\ss}ner for helpful discussions and providing effective computational methods.
\bibliographystyle{alpha}

\noindent
Authors' Address: \\
\noindent
Lehrstuhl Mathematik VIII, Universit\"at Bayreuth\\
Universit\"atsstra\ss e  30, D-95447 Bayreuth\\
E-mail address:\\
binru.li@uni-bayreuth.de\\

\end{document}